\documentclass[a4paper,12pt]{article}

\usepackage{latexsym}
\usepackage{amssymb}
\usepackage{amsmath}
\usepackage{amsthm}
\usepackage{enumerate}
\usepackage[all]{xy}

\def\A{{\mathcal{A}}}
\def\R{{\mathbb{R}}}
\def\K{{\mathbb{K}}}
\def\p{{\partial}}
\def\res{{\mathbf{res}}}
\def\S{{\mathbf{c}\mathrm{Shi}}}
\def\C{{\mathbf{c}\mathrm{Cat}}}
\def\sp{{\mathrm{SRB}_{+}}}
\def\sm{{\mathrm{SRB}_{-}}}

\DeclareMathOperator{\Der}{Der}

\newtheorem{theorem}{Theorem}[section]
\newtheorem{prop}[theorem]{Proposition}
\newtheorem{cor}[theorem]{Corollary}
\newtheorem{lemma}[theorem]{Lemma}
\newtheorem{define}[theorem]{Definition}
\newtheorem{rem}[theorem]{Remark}

\title{
A basis construction of the extended Catalan and Shi arrangements of the type $A_{2}$}

\author{
Takuro Abe\thanks{Department of Mechanical Engineering and Science, Kyoto University, Kyoto 606-8501, Japan. 
e-mail:abe.takuro.4c@kyoto-u.ac.jp} and 
Daisuke Suyama\thanks{Department of Mathematics, Hokkaido University, Sapporo, Hokkaido 060-0810, Japan.
email:dsuyama@math.sci.hokudai.ac.jp
This work was supported by JSPS KAKENHI Grant Number 24$\cdot$6346.}
}

\date{}

\begin{document}

\maketitle

\begin{abstract} 
In \cite{tmultiderivations}, Terao proved the freeness of 
multi-Coxeter arrangements with constant 
multiplicities by giving an explicit construction of bases. Combining it with 
algebro-geometric method, Yoshinaga proved the freeness of 
the extended Catalan and Shi arrangements in \cite{ycharacterization}. 
However, there have been no explicit constructions of 
the bases for the logarithmic derivation modules of 
the extended Catalan and Shi arrangements. 
In this paper, we give the first explicit construction of them 
when the root system is of the type $A_2$. 
\end{abstract} 

{\footnotesize {\it Keywords:} Hyperplane arrangement, Shi arrangement, Catalan arrangement, Free arrangement, Weyl group, Affine Weyl group, Logarithmic derivations}

{\footnotesize {\it 2010 MSC:}  32S22, 20F55, 13N15}

\section{Introduction} 

Let $V$ be an $\ell$-dimensional vector space over a field $\K$.
Let $S=S(V^{*})$ be the symmetric algebra of the dual space $V^{*}$
and $\Der (S)$ the module of derivations
\[
\Der (S)
=\{ \theta :S \rightarrow S \mid \theta \text{ is } \mathbb{K} \text{-linear},\ 
\theta (fg)=\theta (f)g+f\theta (g)\ (f,g \in S) \}.
\]
An (affine) arrangement (of hyperplanes) $\A$ is the finite set of affine hyperplanes in 
$V$. An arrangement is central if every hyperplane in $\A$ is linear. 
For a central arrangement $\A$, the \textbf{logarithmic derivation module} 
$D(\A)$ is defined by 
\begin{align*}
D(\A) &= \{ \theta \in \Der (S) \mid
\theta (Q(\A)) \in Q(\A) S \} \\
&= \{ \theta \in \Der (S) \mid
\theta (\alpha_{H}) \in \alpha_{H} S \text{ for all } H \in \A \},
\end{align*}
where $\alpha_{H}\ (H \in \A)$ is a linear form such that $\ker (\alpha_{H}) = H$
and $Q(\A)$ is the defining polynomial of $\A$, 
that is, $Q(\A)=\prod_{H \in \A} \alpha_{H}$.
When $D(\A)$ is a free $S$-module, 
$\A$ is called a \textbf{free arrangement}.
Then there exists a homogeneous basis 
$\{ \theta_{1}, \ldots ,\theta_{\ell} \}$ for $D(\A)$
and $\exp \A := ( \deg \theta_{1}, \ldots ,\deg \theta_{\ell} )$
is called the \textbf{exponents} of $\A$.   
For an affine arrangement $\A$ in $V$, 
$\mathbf{c}\A$ denotes the \textbf{cone} 
\cite[Definition 1.15]{otarrangements} over $\A$.
The cone $\mathbf{c}\A$ is a central arrangement in
an $(\ell +1)$-dimensional vector space $U$.
We may regard $U^{*}=V^{*} \oplus \langle z \rangle$
by using a new coordinate $z$, 
and let $S_{z}$ denote the symmetric algebra $S(U^{*})$ 
of the dual space $U^{*}$.

Let $E$ be an $\ell$-dimensional Euclidean space 
and $\Phi$ be a crystallographic irreducible root system 
in the dual space $E^{*}$.
Let $\Phi^{+}$ be a positive system of $\Phi$.
For $\alpha \in \Phi^{+}$ and $i \in \mathbb{Z}$,
define the affine hyperplane $H_{\alpha,i}$ by
\[
H_{\alpha,i}:=
\{ v \in V \mid \alpha (v)=i \}.
\]
Then the arrangement 
$\A_{\Phi}=\{ H_{\alpha,0} \mid \alpha \in \Phi^{+} \}$
is called the \textbf{Weyl arrangement} of the type $\Phi$.

\begin{define}
Let $k \in \mathbb{Z}_{\geq 0}$.
Then the \textbf{extended Shi arrangement} $\mathrm{Shi}^{k}$ of the type $\Phi$ 
and the \textbf{extended Catalan arrangement} $\mathrm{Cat}^{k}$ of the type $\Phi$
are affine arrangements defined by
\begin{align*}
\mathrm{Shi}^{k}
&:=\{ H_{\alpha,i} \mid \alpha \in \Phi^{+}, -k+1\leq i \leq k \}, \\
\mathrm{Cat}^{k}
&:=\{ H_{\alpha,i} \mid \alpha \in \Phi^{+}, -k\leq i \leq k \}.
\end{align*}
\end{define}

In particular, 
the arrangement $\mathrm{Shi}^{1}$ is
called the \textbf{Shi arrangement} which was introduced by 
J.-Y. Shi in \cite{skazhdan} in the study of the Kazhdan-Lusztig representation theory 
of the affine Weyl groups.

There are a lot of researches on the freeness of 
the cones over the extended Catalan 
and Shi arrangements. The first breakthrough 
was the proof of the freeness of multi-Coxeter arrangements with constant 
multiplicities by Terao in \cite{tmultiderivations}. 
Combining this result with algebro-geometric methods, Yoshinaga proved 
the freeness of the cones over the extended Catalan
and Shi arrangements in \cite{ycharacterization}. 
However, there have been few researches how to construct their explicit bases. 
Recently, in the case of types $A_{\ell},B_{\ell},C_{\ell},D_{\ell}$, 
explicit bases for the cone $\S^1$ over the Shi arrangements were constructed (\cite{stshi},\cite{sbasis},\cite{gptshi}).
Also, a nice basis for the extended Shi arrangements was determined in \cite{atsimple}. 

In this paper, we give the first explicit construction of a series of 
bases for the extended Catalan and 
Shi arrangements when the corresponding root system is of the type $A_{2}$. 
Namely, we construct bases for the logarithmic modules of these arrangements as follows:

\begin{theorem}
\label{main}
Let $\Phi$ be the root system of the type $A_2$, 
$\{\alpha_1,\alpha_2\}$ a simple system and 
$\{\partial_1,\partial_2\}$ its dual basis for 
$\Der (S)$. 
For $i \in \mathbb{Z}_{\geq 0}$, define
\[
M_{n}=
\begin{pmatrix}
\alpha_{1} +nz & (2\alpha_{1}+4\alpha_{2}+3nz)(\alpha_{1}+nz) \\
\alpha_{2} +nz & -(4\alpha_{1}+2\alpha_{2}+3nz)(\alpha_{2}+nz)
\end{pmatrix},
\]
\begin{align*}
N_{n} 
&=
\begin{pmatrix}
0 & 1 \\
1 & 0
\end{pmatrix}
(M_{n}^T)|_{z\rightarrow -z}
\\
&=
\begin{pmatrix}
(2\alpha_{1}+4\alpha_{2}-3nz)(\alpha_{1}-nz) & -(4\alpha_{1}+2\alpha_{2}-3nz)(\alpha_{2}-nz) \\
\alpha_{1}-nz & \alpha_{2}-nz
\end{pmatrix},
\end{align*}
\[
T_{n}=
\begin{pmatrix}
\dfrac{1}{3n+1} & 0 \\
0 & \dfrac{1}{3n+2}
\end{pmatrix},
\]
\[
A=[I^{*}(\alpha_{i},\alpha_{j})]_{1 \leq i,j \leq 2}=
\begin{pmatrix}
2 & -1 \\
-1 & 2
\end{pmatrix},
\]
where $I^{*}$ is the natural inner product on $E^{*}$
induced from the $W$-invariant inner product $I$ on $E$.
Then the Euler derivation and 
\[
[\partial_{1},\partial_{2}] \prod_{i=0}^{k-1}(M_{i} T_i N_{i+1}A^{-1})
\]
form a basis for $D(\S^{k})$, and
\[
[\partial_{1},\partial_{2}] (\prod_{i=0}^{k-1}(M_{i} T_i N_{i+1}A^{-1}))M_k
\]
a $W$-invariant basis for $D(\C^{k})$.
\end{theorem}

The idea to prove Theorem \ref{main} is to use the simple-root bases (\cite{atsimple}) and 
Terao's matrix $B^{(k)}$ (\cite{tmultiderivations}, \cite{atprimitive}) with the invariant theory. 
Namely, if we fix a simple system and 
a primitive derivation, then we know the existence of a family of nice bases 
(the simple-root basis plus/minus) 
for the logarithmic modules of $\S^k$ for all $k \in \mathbb{Z}_{\geq 0}$. 
In general, we cannot compute the explicit form of simple-root bases. However, 
by computations based on 
invariant theory and Weyl group actions, we can find a way to 
construct the bases for that of $\C^{k}$ from 
simple-root bases. 
by restricting them onto the infinite hyperplane  
and applying the invariant theoretic method, 
we may connect these new bases. Hence starting from the empty arrangement, 
we can construct the required bases for the extended Catalan and 
Shi arrangements inductively. In that invariant theory, Terao's matrix 
$B^{(k)}$ plays the essential role. 

The organization of this paper is as follows.
In section 2, we review the simple-root bases 
for the extended Shi arrangements introduced 
in \cite{atsimple}, which play key roles in our construction of bases. 
In section 3, we give an explicit construction of bases for
the extended Catalan and Shi arrangements of the type $A_{2}$ in 
Theorem \ref{cons} for a certain primitive derivation. Using Theorem \ref{cons}, 
we prove Theorem \ref{main}.

\section{Preliminaries}

In this section we review the definition and properties of multiarrangements and 
the simple-root bases for 
the extended Shi arrangements.

First, let $\A$ be a central arrangement in $V=\mathbb{K}^\ell$,
$\{ x_{1},\ldots, x_{\ell} \}$ a basis for $V^{*}$
and fix $H \in \A$. Define 
$$
D_0(\A):=\{\theta \in D(\A) \mid \theta(\alpha_H)=0\}.
$$
Then it is known (e.g., see \cite[Proposition 4.27]{otarrangements}) that $$
D(\A) = S \theta_E \oplus D_0(\A)
$$
for the Euler derivation $\theta_E:=\sum_{i=1}^\ell x_i \frac{\partial}{\partial x_{i}}$. 
Hence $\A$ is free if and only if $D_0(\A)$ is a free $S$-module, and 
$\theta_2,\ldots,\theta_\ell$ form a basis for $D_0(\A)$ if and only if 
$\theta_E,\theta_2,\ldots,\theta_\ell$ form a basis for $D(\A)$. 
To check the freeness, the following is the most convenient.

\begin{prop}{\bf (Saito's criterion, \cite{stheory})}
Let $\theta_1,\ldots,\theta_\ell \in D(\A)$ and 
$M:=(\theta_i(x_j))$. Then $D(\A)$ is a free $S$-module with 
basis $\theta_1,\ldots,\theta_\ell$ if and only if 
$\det M=c\prod_{H \in \A} \alpha_H$ for some non-zero $c \in \mathbb{K}$.
\label{saito}
\end{prop} 

For a fixed $H \in \A$, let $\A^H:=\{K \cap H \mid K \in \A \setminus \{H\}\}$ and define a map $m_H:
\A^H \rightarrow \mathbb{Z}_{>0}$ by 
$$
m_H(K \cap H):=|\{L \in \A \setminus \{H\} \mid 
L \cap H=K \cap H\}|.
$$
Then for a logarithmic module 
\begin{multline*}
D(\A^H,m_H):=\{\theta \in \mbox{Der}(S/(\alpha_H)) 
\mid 
\theta(\alpha_K) \in (S/(\alpha_H)) (\alpha_K)^{m_H(K)} \\ 
\text{ for any } K \in \A^H \},
\end{multline*}
the Ziegler restriction map $\res :D_0(\A) \rightarrow D(\A^H,m_H)$ 
is defined by $\res (\theta):=\theta|_{\alpha_H=0}$. For 
details, see \cite{zmultiarrangements}. 

\begin{prop}[\cite{zmultiarrangements}]
Assume that $\A \neq \emptyset$ is free with $\exp(\A)=(1,d_2,\ldots,d_\ell)$. 
Then $D_0(\A^H,m_H)$ is also free with basis $\varphi_2,\ldots,\varphi_\ell$ 
such that $\deg (\varphi_i)=d_i\ (i=2,\ldots,\ell)$. Moreover, the Ziegler restriction map 
is surjective.
\label{ziegler}
\end{prop}

For the rest of this section,
let $E$ be an $\ell$-dimensional Euclidean space,
and we recall the simple-root bases introduced in \cite{atsimple}. 
Let $W$ be a finite irreducible reflection group corresponding to
an irreducible root system $\Phi$ in $E^{*}$. Then by the famous 
theorem of Chevalley, there are 
homogeneos basic invariants $P_{1},\ldots ,P_{\ell}$ 
generating the $W$-invariant ring $S^{W}$ of $S$ as $\mathbb{R}$-algebra such that
\[
\deg P_{1} < \deg P_{2} \leq \cdots \leq \deg P_{\ell -1} < \deg P_{\ell}.
\]
Let $F$ be the quotient field of
$S$.
Then the \textbf{primitive derivation} 
$D = \frac{\partial}{\partial P_{\ell}} \in \Der (F)$ 
is characterized by 
\[
D(P_{i})=
\begin{cases}
c \in \mathbb{R}^{\times} & (i=\ell) \\
0 & (1 \leq i \leq \ell -1)\ .
\end{cases}
\]
The primitive derivation $D$ is $W$-invariant and uniquely determined up to
nonzero constant multiple $c$ independent of the choice of 
the basic invariants.
Define an affine connection 
$\nabla :\Der (F) \times \Der (F) \rightarrow \Der (F)$ by
\[
\nabla_{\theta_{1}} \theta_{2} 
= \sum_{i=1}^{\ell} \theta_{1} (f_{i}) 
\frac{\partial}{\partial x_{i}} 
\]
for
$\theta_{1},\theta_{2} \in \Der (F)$
with
$\theta_{2} = \sum_{i=1}^{\ell} f_{i} 
\frac{\partial}{\partial x_{i}}$.
For $m \in \mathbb{Z}_{> 0}$, define an $S$-module $D(\A_{\Phi},m)$ by
\[
D(\A_{\Phi},m) = 
\{ \theta \in \Der (S) \mid 
\theta (\alpha_{H}) \in \alpha_{H}^{m} S \text{ for any } H \in \A_{\Phi} \}.
\]

Note that the action of $W$ onto $E$ canonically extends to those onto 
$E^*,\ S,\ \Der (S)$ and $D(\A_{\Phi},m)$. Let $D(\A_{\Phi},m)^W$ denote 
the $W$-invariant set of $D(\A_\Phi,m)$. 

\begin{lemma}\label{nabla}
{\bf (\cite{yprimitive}, Lemma 9)}
For the derivations 
$\frac{\partial}{\partial P_{i}} \in \Der (S^{W})\ (1\leq i \leq \ell)$,
\[
\nabla_{\frac{\partial}{\partial P_{i}}}
D(\A_{\Phi},2k+1)^{W} \subset D(\A_{\Phi},2k-1)^{W}
\ \ (k>0).
\]
\end{lemma}

In particular, as shown in \cite{atprimitive}, the connection $\nabla_{D}$ induces an 
$\mathbb{R}[P_{1},\ldots ,P_{\ell-1}]$-isomorphism
\[
\nabla_{D}:D(\A_{\Phi},2k+1)^{W} \stackrel{\sim}{\longrightarrow} D(\A_{\Phi},2k-1)^{W}
\ \ (k>0).
\]
So we can consider the inverse map
\[
\nabla_{D}^{-1}:D(\A_{\Phi},2k-1)^{W} 
\stackrel{\sim}{\longrightarrow} D(\A_{\Phi},2k+1)^{W}.
\]

\begin{prop}
{\bf 
(\cite{yprimitive}, Theorem 7)}
Define $\partial_v\ (v \in E)$ by 
$\partial_v(\alpha):=\alpha (v)$ for $\alpha \in E^*$
which is canonically extended to a derivation 
$\p_{v}: S \rightarrow S$. 
Define $\Xi:E \rightarrow D(\A_{\Phi},2k)$ by
$\Xi (v) = \nabla_{\partial_{v}} \nabla_{D}^{-k} \theta_{E}$.
Then $\Xi$ is a $W$-isomorphism.
\label{atsimple}
\end{prop}

\begin{prop}
{\bf (\cite{ycharacterization}, Theorem 1.2)}
Let $D_{0}(\S^{k})=\{ \theta \in D(\S^{k}) \mid \theta(z)=0 \}$.
Then the Ziegler restriction map 
$\res :D_{0}(\S^{k}) \rightarrow D(\A_{\Phi},2k)$
is surjective.
In particular, the morphism 
$\res :D_{0}(\S^{k})_{kh} \rightarrow D(\A_{\Phi},2k)_{kh}$
is an $\mathbb{R}$-linear isomorphism, 
where $D_{0}(\S^{k})_{kh}$ and $D(\A_{\Phi},2k)_{kh}$
are the homogeneous parts of degree $kh$ of 
$D_{0}(\S^{k})$ and $D(\A_{\Phi},2k)$ respectively, 
and $h$ is the Coxeter number of $\Phi$. 
\end{prop}


\begin{define}
{\bf (\cite{atsimple}, Definition 1.2)}
Fix $k \in \mathbb{Z}_{\geq 0}$.
Define a linear isomorphism $\Theta :E \rightarrow D_{0}(\S^{k})$
by $\Theta=\res^{-1} \circ \Xi$.
Let $\{ \alpha_{1},\ldots ,\alpha_{\ell} \} \subset E^{*}$ 
be a simple system of $\Phi^+$ and 
$\{ \alpha_{1}^{*},\ldots ,\alpha_{\ell}^{*} \} \subset E$
be its dual basis. 
Then the derivations 
\[
\varphi_{i}^{(k)} = \Theta(\alpha_{i}^{*})
\ \ (1\leq i\leq \ell )
\]
are called a \textbf{simple-root basis plus} $(\mathbf{SRB}_{+})$
of $D_{0}(\S^{k})$ and the derivations
\[
\psi_{i}^{(k)} 
= \sum_{p=1}^{\ell} I^{*}(\alpha_{i},\alpha_{p}) \varphi_{p}^{(k)}
\ \ (1\leq i\leq \ell )
\]
are called a \textbf{simple-root basis minus} $(\mathbf{SRB}_{-})$
of $D_{0}(\S^{k})$.
Here $I^{*}$ is the natural inner product on $E^{*}$
induced from the inner product $I$ on $E$.
\end{define}

\begin{rem}\label{I^{*}}
Let $\Omega_{F}$ denote the dual $S$-module of $\Der (F)$.
The inner product 
$I^{*}:E^{*} \times E^{*} \rightarrow \R$
can be extended to a nondegenerate $F$-bilinear form
$I^{*}:\Omega_{F} \times \Omega_{F} \rightarrow F$.
Define an $F$-linear isomorphism $I^{*}:\Omega_{F} \rightarrow \Der (F)$
by $I^{*}(\omega)(f) := I^{*}(\omega ,df)$ 
for $\omega \in \Omega_{F}, f \in F$.
Then the restriction of $\sm$ $\res (\psi_{i}^{(k)})$ 
can be expressed as follows:
\begin{align*}
\res (\psi_{i}^{(k)})
&= \sum_{p=1}^{\ell} I^{*}(\alpha_{i},\alpha_{p})
\nabla_{\p_{\alpha_{p}^{*}}} \nabla_{D}^{-k} \theta_{E} \\
&= \nabla_{\sum_{p=1}^{\ell} I^{*}(\alpha_{i},\alpha_{p}) \p_{\alpha_{p}^{*}}} \nabla_{D}^{-k} \theta_{E} \\
&= \nabla_{I^{*}(d \alpha_{i})} \nabla_{D}^{-k} \theta_{E}.
\end{align*}
\end{rem}

\begin{rem}\label{innermatrix}
Let $A=[I^{*}(\alpha_{i},\alpha_{j})]_{1\leq i,j \leq \ell}$
be the inner product matrix. Then by definition, 
an $\sp$ $\{ \varphi_{1}^{(k)}, \ldots ,\varphi_{\ell}^{(k)} \}$ and 
an $\sm$ $\{ \psi_{1}^{(k)}, \ldots ,\psi_{\ell}^{(k)} \}$ 
are related as follows:
\[
[\varphi_{1}^{(k)}, \ldots ,\varphi_{\ell}^{(k)}]
=[\psi_{1}^{(k)}, \ldots ,\psi_{\ell}^{(k)}] A^{-1}.
\]
\end{rem}

It follows from Schur's lemma that these bases are
uniquely determined if we fix a simple system and 
a primitive derivation $D$.
These bases can be characterized by the following conditions:

\begin{prop}\label{charsrb}
{\bf (\cite{atsimple}, Theorem 1.3)}

$(1)$
Let $\varphi_{1}^{(k)},\ldots ,\varphi_{\ell}^{(k)}$ be an $\sp$ 
of $D_{0}(\S^{k})$.
Then $\varphi_{1}^{(k)},\ldots ,\varphi_{\ell}^{(k)}$ satisfy
\[
\varphi_{i}^{(k)}(\alpha_{j}+kz) \in (\alpha_{j}+kz)S_{z}
\ \ (i \neq j).
\]

$(2)$ 
Let $\psi_{1}^{(k)},\ldots ,\psi_{\ell}^{(k)}$ be an $\sm$
of $D_{0}(\S^{k})$.
Then $\psi_{1}^{(k)},\ldots ,\psi_{\ell}^{(k)}$ satisfy
\[
\psi_{i}^{(k)} \in (\alpha_{i}-kz)\Der (S_{z})
\ \ (1 \leq i \leq \ell ).
\]
\end{prop}

Now we introduce some propositions concerning 
the action of $W$ to these bases.

\begin{prop}\label{keuler}
{\bf (\cite{atsimple}, Theorem 4.3)}
The derivation
\[
\sum_{i=1}^{\ell} (\alpha_{i}+kz) \varphi_{i}^{(k)}
\]
is called the \textbf{k-Euler derivation}.
The $k$-Euler derivation is $W$-invariant and
belongs to $D_{0}(\C^{k})_{kh+1}$.
\end{prop}

\begin{prop}\label{reflection}
{\bf (\cite{atsimple}, Theorem 3.5)}
Let $s_{i} \in W$ be the reflection corresponding to $\alpha_{i}$
for $1 \leq i \leq \ell$.
Then 
\begin{enumerate}
\item $s_{i} \varphi_{j}^{(k)}=\varphi_{j}^{(k)}$ whenever $i \neq j$, and 
\item $s_{i} \left( \dfrac{\psi_{i}^{(k)}}{(\alpha_{i}-kz)} \right)
=\dfrac{\psi_{i}^{(k)}}{(\alpha_{i}-kz)}$\ \ for $1 \leq i \leq \ell$.
\end{enumerate}
\end{prop}

\section{Construction of bases of the type $A_{2}$}

For the rest of this paper, 
we assume that the root system $\Phi$ is of the type $A_2$. Hence 
the Coxeter number $h=3$ and Yoshinaga's result 
in \cite{ycharacterization} tells us that 
$\S^k$ and $\C^k$ are free with exponents 
$$
\exp(\S^k)=(1,3k,3k),\ \exp(\C^k)=(1,3k+1,3k+2).
$$
Let $\{ \alpha_{1},\alpha_{2} \} \subset E^{*}$ be a simple system. For $\alpha \in \Phi^+$ and 
$k \in \mathbb{Z}$, 
let $H_{\alpha-kz}:=\{\alpha-kz=0\}$. Then the results in \cite{atsimple} show that 
$\S^k \setminus \{H_{\alpha_i-kz}\}$ and 
$\S^k \setminus \{H_{\alpha_1-kz},H_{\alpha_2-kz}\}$ are both free with exponents 
\begin{eqnarray*}
\exp(\S^k \setminus \{H_{\alpha_i-kz}\})&=&(1,3k-1,3k),\\
\exp(\S^k \setminus \{H_{\alpha_1-kz},H_{\alpha_2-kz}\})&=&(1,3k-1,3k-1)
\end{eqnarray*}
for $i=1,2$. Now we prove the key result to show 
Theorem \ref{main}.

\begin{theorem}
\label{cons}
Let us fix basic invariants 
$$
P_{1}:=\alpha_{1}^{2}+\alpha_{1}\alpha_{2}+\alpha_{2}^{2},\ 
P_{2}:=\frac{2}{27} (\alpha_{1}-\alpha_{2})(\alpha_{1}+2\alpha_{2})(2\alpha_{1}+\alpha_{2})
$$
of the Weyl group $W$ and choose the primitive derivation $D$ in such a way that 
$D(P_2)=\dfrac{1}{3}$. 
For $k \in \mathbb{Z}_{\geq 0}$, let $M_k,\ N_{k}$
and $T_{k}$ be the same as in Theorem \ref{main}. 

Let $\varphi_{1}^{(k)},\varphi_{2}^{(k)}$ be an $\sp$ of $D_{0}(\S^{k})$.
Then 
\[
[\varphi_{1}^{(k)},\varphi_{2}^{(k)}] M_{k}
\]
form a $W$-invariant basis for $D_{0}(\C^{k})$, and
\[
[\varphi_{1}^{(k)},\varphi_{2}^{(k)}] M_{k} T_{k} N_{k+1}
\]
form an $\sm$ of $D_{0}(\S^{k+1})$.
\end{theorem}


We prove Theorem \ref{cons} by using following propositions.

\begin{prop}\label{srb+}
Let $[\theta_{1}^{(k)},\theta_{2}^{(k)}] = [\varphi_{1}^{(k)},\varphi_{2}^{(k)}] M_{k}$.
Then $\theta_{1}^{(k)},\theta_{2}^{(k)}$ form a $W$-invariant basis for $D_{0}(\C^{k})$.
\end{prop}
\begin{proof}
Since 
$\theta_{1}^{(k)}=
(\alpha_{1}+kz)\varphi_{1}^{(k)}+(\alpha_{2}+kz)\varphi_{2}^{(k)}$ 
is the $k$-Euler derivation,
it follows from Proposition \ref{keuler} that
$\theta_{1}^{(k)} \in D_{0}(\C^{k})^{W}$.
Let us show $\theta_{2}^{(k)} \in D_0(\C^{k})^{W}$.
By Proposition \ref{charsrb} $(1)$, it is clear that 
$\theta_{2}^{(k)}(\alpha_{i}+kz) \in (\alpha_{i}+kz)S_{z}$ $(i=1,2)$.
Since
\begin{align*}
\theta_{2}^{(k)} 
&= (2\alpha_{1}+4\alpha_{2}+3kz)(\alpha_{1}+kz) \varphi_{1}^{(k)}
-(4\alpha_{1}+2\alpha_{2}+3kz)(\alpha_{2}+kz) \varphi_{2}^{(k)} \\
&= (2\alpha_{1}+4\alpha_{2}+3kz)\{ \theta_{1}^{(k)} - (\alpha_{2}+kz)\varphi_{2}^{(k)} \} \\
&\hspace{15em}
-(4\alpha_{1}+2\alpha_{2}+3kz)(\alpha_{2}+kz) \varphi_{2}^{(k)} \\
&= (2\alpha_{1}+4\alpha_{2}+3kz)\theta_{1}^{(k)}
-6(\alpha_{1}+\alpha_{2}+kz)(\alpha_{2}+kz)\varphi_{2}^{(k)},
\end{align*}
it holds that 
$\theta_{2}^{(k)}(\alpha_{1}+\alpha_{2}+kz) \in (\alpha_{1}+\alpha_{2}+kz)S_{z}$. 
So $\theta_{2}^{(k)} \in D_0(\C^{k})$.
Moreover, since $s_{i}\varphi_{j}^{(k)} = \varphi_{j}^{(k)}$ $(i \neq j)$ for the reflection $s_{i}$ 
corresponding to $\alpha_{i}$ 
because of Proposition \ref{reflection} $(1)$,
\begin{eqnarray*}
s_{1} \theta_{2}^{(k)}
&=& (2\alpha_{1}+4\alpha_{2}+3kz)s_{1}\theta_{1}^{(k)}
-6(\alpha_{2}+kz)(\alpha_{1}+\alpha_{2}+kz)s_{1}\varphi_{2}^{(k)} \\
&=&\theta_{2}^{(k)}.
\end{eqnarray*}
Similarly, we can express 
$\theta_2^{(k)}$ in terms of $\theta_1^{(k)}$ and $\varphi_1^{(k)}$. Then the 
same argument as the above 
shows that $s_{2}\theta_{2}^{(k)}=\theta_{2}^{(k)}$. 
Hence $\theta_{2}^{(k)}$ is $W$-invariant.
Finally, since 
\[
\det (M_{k}) = -6(\alpha_{1}+kz)(\alpha_{2}+kz)(\alpha_{1}+\alpha_{2}+kz),
\]
and $\varphi_1^{(k)},\varphi_2^{(k)}$ form a basis for $D_0(\S^k)$, 
Proposition \ref{saito} shows that $\theta_1^{(k)},\theta_2^{(k)}$ form a basis for $D_0(\C^k)$.
\end{proof}

\begin{lemma}\label{nabla2}
Let $\Omega^{1}(\A_{\Phi})$ denote the module of 
logarithmic differential forms of $\A_{\Phi}$ (i.e., 
the dual $S$-module of $D(\A_\Phi)$).
If $\omega \in \Omega^{1}(\A_{\Phi})$, then 
$\nabla_{I^{*}(\omega)}\nabla_{D}^{-k}\theta_{E} \in D(\A_{\Phi},2k-1)$. 
\end{lemma}
\begin{proof}
By Definition 3.1 and Theorem 3.3 in \cite{atprimitive}, it follows that 
\[
I^{*}(\Omega^1(\A_{\Phi}))
\subset \bigoplus_{i=1}^{\ell} S \frac{\partial}{\partial P_{i}}.
\]
Since $\nabla_{\frac{\partial}{\partial P_{i}}}
\nabla_{D}^{-k}\theta_{E} \in D(\A_{\Phi},2k-1)$
by Lemma \ref{nabla}, we conclude that 
$\nabla_{I^{*}(\omega)}\nabla_{D}^{-k}\theta_{E} \in D(\A_{\Phi},2k-1)$.
\end{proof}

\begin{prop}\label{srb-}
Let $\psi_{1}^{(k)},\psi_{2}^{(k)}$ be an $\sm$ of $D_{0}(\S^{k})$. 
Then 
$[\eta_{1}^{(k-1)},\eta_{2}^{(k-1)}] := [\psi_{1}^{(k)},\psi_{2}^{(k)}] N_{k}^{-1}$
form a $W$-invariant basis for $D_{0}(\C^{k-1})$.
\end{prop}
\begin{proof}
First we will show that $\eta_{1}^{k-1} \in D_{0}(\C^{k-1})^{W}$.
Since
\[
N_{k}^{-1}=
\begin{pmatrix}
\dfrac{1}{6(\alpha_{1}-kz)(\alpha_{1}+\alpha_{2}-kz)} &
\dfrac{4\alpha_{1}+2\alpha_{2}-3kz}{6(\alpha_{1}-kz)(\alpha_{1}+\alpha_{2}-kz)} \\
& \\
-\dfrac{1}{6(\alpha_{2}-kz)(\alpha_{1}+\alpha_{2}-kz)}
&
\dfrac{2\alpha_{1}+4\alpha_{2}-3kz}{6(\alpha_{2}-kz)(\alpha_{1}+\alpha_{2}-kz)}
\end{pmatrix},
\]
we have
\[
\eta_{1}^{(k-1)}=
\frac{1}{6(\alpha_{1}+\alpha_{2}-kz)}
\left(
\frac{\psi_{1}^{(k)}}{\alpha_{1}-kz}
-\frac{\psi_{2}^{(k)}}{\alpha_{2}-kz}
\right).
\]
Consider a commtative diagram
\[
\xymatrix{
D_{0}(\S^{k} \setminus \{ H_{\alpha_{1}-kz},H_{\alpha_{2}-kz} \} )_{3k-1}
\ar[r]^-{\res}_-{\sim} \ar@{}[d]|*{\bigcup}
& D(\A_\Phi, 2k-{\bf m})_{3k-1} \ar@{}[d]|*{\bigcup} \\
(\alpha_{1}+\alpha_{2}-kz) D_{0}(\C^{k-1})_{3k-2}
\ar[r]^-{\res}_-{\sim}
& (\alpha_{1}+\alpha_{2}) D(\A_\Phi, 2k-1)_{3k-2},
}
\]
where ${\bf m}:\A_{\Phi} \rightarrow \{ 0,1 \}$ is a multiplicity defined by
\[
{\bf m}(H)=
\begin{cases}
1 & H \in \{ H_{\alpha_{1}},H_{\alpha_{2}} \} \\
0 & H=H_{\alpha_{1}+\alpha_{2}}
\end{cases}
\hspace{1em}
( H \in \A_{\Phi} ).
\]
Let
\[
\eta := 6(\alpha_{1}+\alpha_{2}-kz)\eta_{1}^{(k-1)}
=\frac{\psi_{1}^{(k)}}{\alpha_{1}-kz}-\frac{\psi_{2}^{(k)}}{\alpha_{2}-kz}. 
\]
It follows from Proposition \ref{charsrb} $(2)$ that
$\eta \in D_{0}(\S^{k} \setminus \{ H_{\alpha_{1}-kz},H_{\alpha_{2}-kz} \})_{3k-1}$.
By the definition of $\sm$ and Remark \ref{I^{*}}, we have
\begin{align*}
\frac{1}{\alpha_{1}+\alpha_{2}}\ \res (\eta) 
&= \frac{1}{\alpha_{1}+\alpha_{2}}
\ \res 
\left(
\frac{\psi_{1}^{(k)}}{\alpha_{1}-kz}-\frac{\psi_{2}^{(k)}}{\alpha_{2}-kz}
\right) \\
&= \frac{1}{\alpha_{1}+\alpha_{2}}
\left(
\frac{\nabla_{I^{*}(d\alpha_{1})} \nabla_{D}^{-k} \theta_{E}}{\alpha_{1}}
-\frac{\nabla_{I^{*}(d\alpha_{2})} \nabla_{D}^{-k} \theta_{E}}{\alpha_{2}}
\right) \\
&= \nabla_{I^{*} \left( \frac{1}{\alpha_{1}+\alpha_{2}} ( \frac{d\alpha_{1}}{\alpha_{1}} - \frac{d\alpha_{2}}{\alpha_{2}} ) \right)} 
\nabla_{D}^{-k} \theta_{E}.
\end{align*}
Since
\[ 
\frac{1}{\alpha_{1}+\alpha_{2}} ( \frac{d\alpha_{1}}{\alpha_{1}} - \frac{d\alpha_{2}}{\alpha_{2}} )
\in \Omega^{1} (\A_{\Phi}),
\]
Lemma \ref{nabla2} implies that 
\[
\frac{1}{\alpha_{1} + \alpha_{2}} \res (\eta)
\in D(\A_{\Phi},2k-1)_{3k-2}.
\]
Hence
\[
\res (\eta)
\in (\alpha_{1} + \alpha_{2}) D(\A_{\Phi},2k-1)_{3k-2}.
\]
Then we can see that 
$\eta \in (\alpha_{1}+\alpha_{2}-kz) D_{0}(\C^{k-1})_{3k-2}$ 
by chasing the diagram above. 
Thus we may conclude that $\eta_{1}^{(k-1)} \in D_{0}(\C^{k-1})_{3k-2}$.
Since $D_{0}(\C^{k-1})_{3k-2}=D_{0}(\C^{k-1})_{3k-2}^{W}$ is a 
one-dimensional $\mathbb{R}$-vector space 
generated by $(k-1)$-Euler derivation by Proposition \ref{keuler} and 
$\exp(\C^{k-1})=(1,3k-2,3k-1)$, 
we obtain
$\eta_{1}^{(k-1)} \in D_{0}(\C^{k-1})^{W}$.
Next we will prove that 
$\eta_{2}^{(k-1)} \in D_{0}(\C^{k-1})^{W}$.
We compute
\begin{align*}
\eta_{2}^{(k-1)}
&=\frac{4\alpha_{1}+2\alpha_{2}-3kz}{6(\alpha_{1}-kz)(\alpha_{1}+\alpha_{2}-kz)}
\psi_{1}^{(k)}
+\frac{2\alpha_{1}+4\alpha_{2}-3kz}{6(\alpha_{2}-kz)(\alpha_{1}+\alpha_{2}-kz)}
\psi_{2}^{(k)} \\
&=(4\alpha_{1}+2\alpha_{2}-3kz)
\left(
\eta_{1}^{(k-1)}
+\frac{\psi_{2}^{(k)}}{6(\alpha_{2}-kz)(\alpha_{1}+\alpha_{2}-kz)}
\right) \\
&\hspace{14em}
+\frac{2\alpha_{1}+4\alpha_{2}-3kz}{6(\alpha_{2}-kz)(\alpha_{1}+\alpha_{2}-kz)}
\psi_{2}^{(k)} \\
&=(4\alpha_{1}+2\alpha_{2}-3kz)\eta_{1}^{(k-1)}
+\frac{\psi_{2}^{(k)}}{\alpha_{2}-kz}.
\end{align*}
Since 
$\psi_{2}^{(k)}/(\alpha_{2}-kz) \in D_{0}(\S^{k} \setminus \{ H_{\alpha_{2}-kz} \})
\subset D_{0}(\C^{k-1})$, 
it holds that 
$\eta_{2}^{(k-1)} \in D_{0}(\C^{k-1})$.
Moreover, since 
$s_{i} (\psi_{i}^{(k)}/(\alpha_{i}-kz)) =(\psi_{i}^{(k)}/(\alpha_{i}-kz))$
for the reflection $s_{i}$ corresponding to $\alpha_{i}$
because of Proposition \ref{reflection} $(2)$,
\begin{align*}
s_{2}\eta_{2}^{(k-1)}
&=s_{2}(4\alpha_{1}+2\alpha_{2}-3kz) \cdot s_{2}\eta_{1}^{(k-1)}
+s_{2} \left( \frac{\psi_{2}^{(k)}}{\alpha_{2}-kz} \right) \\
&=(4\alpha_{1}+2\alpha_{2}-3kz) \eta_{1}^{(k-1)}+\frac{\psi_{2}^{(k)}}{\alpha_{2}-kz}
=\eta_{2}^{(k-1)}.
\end{align*}
Similarly, we can express $\eta_2^{(k-1)}$ in terms of 
$\eta_1^{(k-1)}$ and $\psi_1^{(k)}/(\alpha_1-kz)$. Then the same argument as  the above 
shows that $s_{1}\eta_{2}^{(k-1)}=\eta_{2}^{(k-1)}$.
Hence $\eta_{2}^{(k-1)}$ is $W$-invariant.
Finally, since 
\[
\det (N_{k}^{-1}) 
=\frac{1}{6(\alpha_{1}-kz)(\alpha_{2}-kz)(\alpha_{1}+\alpha_{2}-kz)},
\]
and $\psi_1^{(k)},\psi_2^{(k)}$ form a basis for 
$D_0(\S^k)$, Proposition \ref{saito} shows that 
$\eta_1^{(k-1)},\eta_2^{(k-1)}$ form a basis for 
$D_0(\C^{k-1})$. 
\end{proof}

It follows from Proposition \ref{srb+} and Proposition \ref{srb-}
that both $[\varphi_{1}^{(k)},\varphi_{2}^{(k)}]M_{k}$
and $[\psi_{1}^{(k+1)},\psi_{2}^{(k+1)}]N_{k+1}^{-1}$ 
are bases for $D_{0}(\C^{k})^{W}$
and their degrees are equal to $( 3k+1,3k+2 )$.
Therefore, there exists a matrix 
$T_{k} \in M_{2}(\mathbb{R}[\alpha_{1},\alpha_{2},z])$
such that 
$[\varphi_{1}^{(k)},\varphi_{2}^{(k)}]M_{k} \cdot T_{k}
=[\psi_{1}^{(k+1)},\psi_{2}^{(k+1)}]N_{k+1}^{-1}
=[\varphi_{1}^{(k+1)},\varphi_{2}^{(k+1)}]AN_{k+1}^{-1}$.
To study this matrix $T_k$, let us show the following lemmas
and proposition.

\begin{lemma}\label{tauands0sm}
Let $\tau$ be the reflection corresponding to $z$ and
$s_{0}$ that to $\alpha_{1}+\alpha_{2}$.
Let $\psi_{1}^{(k)},\psi_{2}^{(k)}$ be an $\sm$ of $D_{0}(\S^{k})$.
Then $\tau s_{0} (\psi_{1}^{(k)})=-\psi_{2}^{(k)}$ and
$\tau s_{0} (\psi_{2}^{(k)})=-\psi_{1}^{(k)}$.
\end{lemma}
\begin{proof}
First, note that $s_{0}(\alpha_{1})=-\alpha_{2}$,
$s_{0}(\alpha_{2})=-\alpha_{1}$ 
and $s_{0}(\alpha_{1}+\alpha_{2})=-(\alpha_{1}+\alpha_{2})$. Also, 
$\tau s_0=s_0 \tau$ since $\tau$ acts on $E$ and $S$ trivially. 
Since $\tau s_{0}$ preserves the Shi arrangement $\S^{k}$,
it holds that $\tau s_{0} (D_{0}(\S^{k}))=D_{0}(\S^{k})$.
Therefore, $\tau s_{0} (\psi_{1}^{(k)}) \in D_{0}(\S^{k})$.
Moreover, 
\[
\frac{\tau s_{0}(\psi_{1}^{k})}{\alpha_{2}-kz}
=-\tau s_{0}
\left(
\frac{\psi_{1}^{k}}{\alpha_{2}-kz}
\right)
\in \Der (S_{z}).
\]
Hence Proposition \ref{charsrb} (2) 
shows that $\tau s_{0} (\psi_{1}^{(k)}) = c \psi_{2}^{(k)}$
for some non-zero $c \in \R^{\times}$.
Since
\begin{align*}
c \nabla_{I^{*}(d\alpha_{2})} \nabla_{D}^{-k} \theta_{E}
=c\psi_{2}^{(k)} |_{z=0} 
&=\tau s_{0} (\psi_{1}^{(k)}) |_{z=0} \\
&= \tau s_{0} (\nabla_{I^{*}(d\alpha_{1})} \nabla_{D}^{-k} \theta_{E}) \\
&= \nabla_{s_{0}(I^{*}(d\alpha_{1}))} \nabla_{D}^{-k} \theta_{E} \\
&= -\nabla_{I^{*}(d\alpha_{2})} \nabla_{D}^{-k} \theta_{E},
\end{align*}
hence Proposition \ref{atsimple} shows that $c=-1$, which implies that 
$\tau s_{0} (\psi_{1}^{(k)})=-\psi_{2}^{(k)}$.
Since $\tau s_0=s_0 \tau$ is a reflection, 
we obtain $\tau s_{0} (\psi_{2}^{(k)})=-\psi_{1}^{(k)}$.	
\end{proof}

\begin{lemma}\label{tauands0sp}
Let $\varphi_{1}^{(k)},\varphi_{2}^{(k)}$ be an $\sp$ of $D_{0}(\S^{k})$.
Then $\tau s_{0} (\varphi_{1}^{(k)})=-\varphi_{2}^{(k)}$ and
$\tau s_{0} (\varphi_{2}^{(k)})=-\varphi_{1}^{(k)}$.
\end{lemma}
\begin{proof}
By Remark \ref{innermatrix} and Lemma \ref{tauands0sm},
we may compute
\begin{align*}
\tau s_{0} [\varphi_{1}^{(k)},\varphi_{2}^{(k)}]
&=\tau s_{0} [\psi_{1}^{(k)},\psi_{2}^{(k)}]A^{-1} \\
&=[\psi_{1}^{(k)},\psi_{2}^{(k)}]
\begin{pmatrix}
0 & -1 \\
-1 & 0 
\end{pmatrix}
A^{-1} \\
&=[\varphi_{1}^{(k)},\varphi_{2}^{(k)}]A
\begin{pmatrix}
0 & -1 \\
-1 & 0 
\end{pmatrix}
A^{-1} \\
&=[\varphi_{1}^{(k)},\varphi_{2}^{(k)}]
\begin{pmatrix}
0 & -1 \\
-1 & 0 
\end{pmatrix},
\end{align*}
which completes the proof.
\end{proof}

\begin{prop}\label{tauinv}
The bases $\theta_1^{(k)},\theta_2^{(k)}$ for $D_0(\C^k)$ and 
$\eta_1^{(k-1)},\eta_2^{(k-1)}$ for \sloppy{$D_0(\C^{(k-1)})$} are 
$W$ and $\tau$-invariant.
\label{inv}
\end{prop}
\begin{proof}
The $W$-invariance is checked in Propositions \ref{srb+} and 
\ref{srb-}.
First we show the $\tau$-invariance of
$\theta_{1}^{(k)}$ and $\eta_{1}^{(k-1)}$.
Note that the action of $\tau$ preserves $\C^{k}$. 
Hence $\tau$ acts on $D_0(\C^{k})$ with the degree preserving. 
By Propositions \ref{srb+} and \ref{srb-}, 
we know that $\dim_{\mathbb{R}} D_0(\C^k)_{3k+1}=
\dim_{\mathbb{R}} D_0(\C^{k-1})_{3k-2}=1$, and 
they are generated by $\theta_1^{(k)}$ and $\eta_1^{(k-1)}$ respectively. 
Hence $\tau \theta_1^{(k)}=c_1\theta_1^{(k)}$ and 
$\tau \eta_1^{(k-1)} =c_2 \eta_1^{(k-1)}$ 
for some non-zero $c_1,c_2 \in \R^{\times}$. 
Since $\theta_1^{(k)}|_{z=0}$ and $\eta_1^{(k-1)}|_{z=0}$ 
are $\tau$-invariant by 
Propositions \ref{srb+} and \ref{srb-}, it holds that $c_1=c_2=1$. 
Hence $\theta_{1}^{(k)}$ and $\eta_{1}^{(k-1)}$
are $\tau$-invariant.
To show the $\tau$-invariance of
$\theta_{2}^{(k)}$ and $\eta_{2}^{(k-1)}$,
it suffices to show that 
$\tau s_{0}(\theta_{2}^{(k)})=\theta_{2}^{(k)}$,
$\tau s_{0}(\eta_{2}^{(k)})=\eta_{2}^{(k)}$ 
because $\theta_{2}^{(k)}, \eta_{2}^{(k-1)}$ are $W$-invariant.
By using Lemma \ref{tauands0sm} and \ref{tauands0sp}, we may compute
\begin{align*}
&\tau s_{0} (\theta_{2}^{(k)}) \\
&=\tau s_{0}
\left(
(2\alpha_{1}+4\alpha_{2}+3kz)(\alpha_{1}+kz)\varphi_{1}^{(k)}
-(4\alpha_{1}+2\alpha_{2}+3kz)(\alpha_{2}+kz)\varphi_{2}^{(k)}
\right) \\
&=(-2\alpha_{2}-4\alpha_{1}-3kz)(-\alpha_{2}-kz)(-\varphi_{2}^{(k)}) \\
&\hspace{14em}
-(-4\alpha_{2}-2\alpha_{1}-3kz)(-\alpha_{1}-kz)(-\varphi_{1}^{(k)}) \\
&=\theta_{2}^{(k)},
\end{align*}
\begin{align*}
&\tau s_{0} (\eta_{2}^{(k)}) \\
&=\tau s_{0}
\left(
\frac{4\alpha_{1}+2\alpha_{2}-3kz}{6(\alpha_{1}-kz)(\alpha_{1}+\alpha_{2}-kz)}\psi_{1}^{(k)}
+\frac{2\alpha_{1}+4\alpha_{2}-3kz}{6(\alpha_{2}-kz)(\alpha_{1}+\alpha_{2}-kz)}\psi_{2}^{(k)}
\right) \\
&=\frac{-4\alpha_{2}-2\alpha_{1}+3kz}{6(-\alpha_{2}+kz)(-\alpha_{2}-\alpha_{1}+kz)}(-\psi_{2}^{(k)}) \\
&\hspace{14em}
+\frac{-2\alpha_{2}-4\alpha_{1}+3kz}{6(-\alpha_{1}+kz)(-\alpha_{2}-\alpha_{1}+kz)}(-\psi_{1}^{(k)}) \\
&=\eta_{2}^{(k)}.
\end{align*}
Hence $\theta_{2}^{(k)}$ and $\eta_{2}^{(k)}$ 
are $\tau$-invariant.
\end{proof}

Now let us study the entries of $T_k$. 
Note that 
every entry of $T_k$ is $W$-invariant since $T_k$ gives a 
transformation between the $W$-invariant bases in $D_0(\C^k)^W$. 
Comparing the degrees of both sides,
we can see that the $(2,1)$-entry of $T_{k}$ is 0,
the $(1,1)$-entry and the $(2,2)$-entry of $T_{k}$ are constants,
and the $(1,2)$-entry of $T_{k}$ is a $W$-invariant polynomial of degree 1.
However, 
$\mathbb{R}[\alpha_{1},\alpha_{2},z]^{W}$ is generated by 
$z$. Hence the $(1,2)$-entry is $c z$ for $c \in \mathbb{R}$. Now 
apply Proposition \ref{tauinv} to conclude that $c=0$.

Hence we may assume that
\[
T_{k}=
\begin{pmatrix}
a_{k} & 0\\
0 & b_{k}
\end{pmatrix}
\ (a_{k},b_{k} \in \mathbb{R}).
\]
Thus $T_k|_{z=0}=T_k$ and 
$[\varphi_{1}^{(k)},\varphi_{2}^{(k)}]|_{z=0}M_{k}|_{z=0} \cdot T_{k}
=[\varphi_{1}^{(k+1)},\varphi_{2}^{(k+1)}]|_{z=0}AN_{k+1}^{-1}|_{z=0}$.
Now recall the following:

\begin{theorem}\label{rest}
{\bf (\cite{atprimitive}, Proposition 4.2)}
Define
\[
R_{2k} := (-1)^{k} J(D^{k}(\alpha_{1}),D^k(\alpha_{2}))^{-1},
\]
where $J(f_1,f_2)$ denotes the Jacobian matrix of $f_1,f_2 \in S$ with 
respect to the simple system $\alpha_1,\alpha_2$, i.e., 
$J(f_1,f_2)=(\partial f_j/\partial \alpha_i)$. 
Then 
\[
[\varphi_{1}^{(k)}|_{z=0},\varphi_{2}^{(k)}|_{z=0}]
=[\nabla_{\partial_{1}} \nabla_{D}^{-k}\theta_{E},
\nabla_{\partial_{2}} \nabla_{D}^{-k}\theta_{E}]
=[\partial_{1},\partial_{2}] AR_{2k}A^{-1}.
\]
\end{theorem}

By using these two, let us compute $T_k$ directly in terms of 
$D(\A_{\Phi},2k+1)$. 
For that purpose, let us rewrite several polynomials and matrices in 
\cite{atsimple} in terms of $\alpha_1$ and $\alpha_2$. First, it is easy to check that 
\[
P_{1}=\alpha_{1}^{2}+\alpha_{1}\alpha_{2}+\alpha_{2}^{2},\ 
P_{2}=\frac{2}{27} (\alpha_{1}-\alpha_{2})(\alpha_{1}+2\alpha_{2})(2\alpha_{1}+\alpha_{2}).
\]
are basic invariants of the type $A_{2}$.
Then the Jacobian matrix $J=J(P_1,P_2)$ is
\[
J=
\begin{pmatrix}
2\alpha_{1}+\alpha_{2} & 
\dfrac{2}{9}(2\alpha_{1}^{2}+2\alpha_{1}\alpha_{2}-\alpha_{2}^{2}) \\
\\
\alpha_{1}+2\alpha_{2} & 
\dfrac{2}{9}(\alpha_{1}^{2}-2\alpha_{1}\alpha_{2}-2\alpha_{2}^{2})
\end{pmatrix}.
\]
Hence 
the primitive derivation $D$ is expressed as
\begin{align*}
D &=\frac{1}{Q}
\begin{vmatrix}
\partial_{1}(P_{1}) & \partial_{1}\\
\partial_{2}(P_{1}) & \partial_{2}
\end{vmatrix} \\
&\doteq
\frac{1}{6 \alpha_{1}\alpha_{2}(\alpha_{1}+\alpha_{2})}
[
(\alpha_{1}+2\alpha_{2}) \partial_{1}
-(2\alpha_{1}+\alpha_{2}) \partial_{2}
],
\end{align*}
where $Q=\alpha_1\alpha_2(\alpha_1+\alpha_2)$ is the defining polynomial of 
the Weyl arrangement of the type $A_{2}$. Also in the above, 
we multiplied $-1/6$ to $D$ to satisfy the condition $D(P_2)=1/3$ in Theorem \ref{cons}. 
For a matrix $M=(m_{ij})$, let $D[M]:=(D(m_{ij}))$. Then we can compute 
\[
D[J]=
\frac{1}{18 \alpha_{1}\alpha_{2}(\alpha_{1}+\alpha_{2})}
\begin{pmatrix}
9\alpha_{2} & 4\alpha_{2} (2\alpha_{1}+\alpha_{2}) \\
-9\alpha_{1} & 4\alpha_{1} (\alpha_{1}+2\alpha_{2})
\end{pmatrix}. 
\]
Moreover, the matrix $B:=J^{T}AD[J]$ and $B^{(k)}:=kB+(k-1)B^{T}$ are 
also computed as follows:
\[
B=
\begin{pmatrix}
0 & 2 \\
1 & 0
\end{pmatrix},
\ 
B^{(k)}=
\begin{pmatrix}
0 & 3k-1 \\
3k-2 & 0
\end{pmatrix}
\]
Hence 
\[
(B^{(k)})^{-1}=
\begin{pmatrix}
0 & \dfrac{1}{3k-2} \\
\dfrac{1}{3k-1} & 0
\end{pmatrix}.
\]
Now by using Theorem \ref{rest}, we can determine the matrix $T_{k}$.

\begin{prop}
\[
T_{k}=
\begin{pmatrix}
\dfrac{1}{3k+1} & 0 \\
0 & \dfrac{1}{3k+2}
\end{pmatrix}.
\]
\label{t}
\end{prop}

\begin{proof}
First recall that 
\[
[\varphi_{1}^{(k)},\varphi_{2}^{(k)}]M_{k} T_{k}
=[\varphi_{1}^{(k+1)},\varphi_{2}^{(k+1)}]AN_{k+1}^{-1}.
\]
Restricting the equality above onto $z=0$
and applying Theorem \ref{rest}, 
we obtain
\[
AR_{2k}A^{-1}(M_{k}|_{z=0})(T_{k}|_{z=0})
=AR_{2k+2}A^{-1}A(N_{k+1}|_{z=0})^{-1}.
\]
Therefore, 
\[T_k=
T_{k}|_{z=0}=
(M_{k}|_{z=0})^{-1}AR_{2k}^{-1}R_{2k+2}(N_{k+1}|_{z=0})^{-1}.
\]
By Proposition 2.6 in \cite{atprimitive},
\[
R_{2k}^{-1}R_{2k+2}=J(B^{(k+1)})^{-1}J^{T}A.
\]
Now we can compute $T_k=T_{k}|_{z=0}$ directly as follows:
\begin{align*}
T_k=T_{k}|_{z=0} &=
(M_{k}|_{z=0})^{-1}AJ(B^{(k+1)})^{-1}J^{T}A(N_{k+1}|_{z=0})^{-1} \\
&=
\begin{pmatrix}
\dfrac{1}{3k+1} & 0 \\
0 & \dfrac{1}{3k+2}
\end{pmatrix}.
\end{align*}

\end{proof}

\noindent
\textit{Proof of Theorem \ref{cons}}. 
Combine Propositions \ref{srb+}, \ref{srb-} and \ref{t}. 
\medskip

\noindent
\textit{Proof of Theorem \ref{main}}. 
First, note that $P_1$ and $P_2$ are unique up to nozero-constant when $\Phi$ is of the type 
$A_2$ since there is no $W$-invariant polynomial of degree one. 
Therefore, the construction in Theorem \ref{cons} shows that 
for any choice of $P_1,\ P_2$ and $D$, the bases constructed by them are unique up to nonzero constants.
Moreover, we can connect the $\sp$ and $\sm$ using 
the inner product matrix $A$ as Remark \ref{innermatrix}.
Hence we may apply Theorem \ref{cons} starting from $[\partial_1,\partial_2]$ inductively to 
obtain the bases stated in Theorem \ref{main}, which completes the proof.

\providecommand{\bysame}{\leavevmode\hbox to3em{\hrulefill}\thinspace}
\providecommand{\MR}{\relax\ifhmode\unskip\space\fi MR }
\providecommand{\MRhref}[2]{%
  \href{http://www.ams.org/mathscinet-getitem?mr=#1}{#2}
}
\providecommand{\href}[2]{#2}

\end{document}